\documentclass[12pt]{amsart}
\usepackage{amssymb}
\newcommand{\A}{\mathcal{A}}
\newcommand{\ind}{\mathsf{ind}}
\newcommand{\End}{\operatorname{End}}
\newcommand{\Rep}{\mathsf{Rep}}
\newcommand{\K}{\mathbb{K}}
\newcommand{\Coh}{\operatorname{Coh}}
\newcommand{\C}{\mathbb{C}}
\newcommand{\Z}{\mathbb{Z}}
\newcommand{\SB}{\mathsf{SB}}
\newcommand{\J}{\mathcal{J}}
\newcommand{\g}{\mathfrak{g}}
\newcommand{\U}{\mathcal{U}}
\newcommand{\Walg}{\mathcal{W}}
\newcommand{\Grk}{\operatorname{Grk}}
\newcommand{\gr}{\operatorname{gr}}
\newcommand{\Orb}{\mathbb{O}}
\newcommand{\Irr}{\operatorname{Irr}}
\newcommand{\Prim}{\operatorname{Prim}}
\newcommand{\slf}{\mathfrak{sl}}

\newcommand{\h}{\mathfrak{h}}
\newcommand{\q}{\mathfrak{q}}

\newcommand{\Afrak}{\mathfrak{A}}

\newcommand{\Weyl}{\mathbb{A}}

\newcommand{\jet}{\mathsf{J}^\infty}
\newcommand{\Str}{\mathcal{O}}

\newcommand{\I}{\mathcal{I}}
\newcommand{\tf}{\mathfrak{t}}
\newcommand{\VA}{\operatorname{V}}
\newcommand{\Ca}{\mathsf{C}}
\newtheorem{Thm}{Theorem}[section]
\newtheorem{Prop}[Thm]{Proposition}

\newtheorem{Lem}[Thm]{Lemma}
\theoremstyle{definition}
\newtheorem{Ex}[Thm]{Example}
\newtheorem{defi}[Thm]{Definition}

\newtheorem{Conj}[Thm]{Conjecture}
\title{Goldie ranks of primitive ideals and  indexes of equivariant Azumaya algebras}
\author{Ivan Losev and Ivan Panin}
\address{I.L.: Department
of Mathematics, Yale University, New Haven, CT,  USA}
\email{ivan.loseu@gmail.com}
\address{I.P: St. Petersburg branch of V.A. Steklov Mathematical Institute,
St. Petersburg, Russian Federation}
\email{paniniv@gmail.com}
\begin{document}
\begin{abstract}
Let $\mathfrak{g}$ be a semisimple Lie algebra.
We establish a new relation between the Goldie rank of a primitive ideal
$\mathcal{J}\subset U(\g)$ and the dimension of the corresponding irreducible
representation $V$ of an appropriate finite W-algebra. Namely, we show that $\operatorname{Grk}(\mathcal{J})
\leqslant \dim V/d_V$, where $d_V$ is the index of a suitable equivariant Azumaya
algebra on a homogeneous space. We also compute $d_V$ in representation theoretic terms.
\end{abstract}
\maketitle
\markright{GOLDIE RANKS AND INDEXES}
\section{Introduction}
In this paper we find a new relation between the Goldie ranks of primitive ideals and the dimensions
of finite dimensional irreducible modules over finite W-algebras.

Let $\g$ be a semisimple Lie algebra over $\C$, and let $\U$ denote its universal enveloping algebra.
Recall that by a primitive ideal in $\U$ one means the annihilator of an irreducible module.
Let $\J$ be a primitive ideal. Then $\U/\J$ is a prime Noetherian algebra and so, by the Goldie theorem (see,
e.g., \cite[Chapter 2]{MR}), it has the full fraction ring, $\operatorname{Frac}(\U/\J)$ that is a matrix
algebra over a skew-field. The rank of this matrix algebra  is called the {\it Goldie rank} of $\J$
and is denoted by $\Grk(\J)$. For example, when $\J$ is the annihilator of a finite dimensional
irreducible representation, then the Goldie rank is the dimension of that representation.
Finding a formula for $\Grk(\J)$ in the general case is a well-known open problem in Lie representation theory.

Recall that to the primitive ideal $\J$ one can assign a nilpotent orbit in $\g$: the unique dense orbit in the
subvariety $\VA(\J)\subset \g$ defined by $\gr\J$. Let $\Orb$ denote this orbit. From $(\g,\Orb)$ one can construct
an associative algebra known as the finite W-algebra, see \cite{Premet1,W_quant}. We denote this
algebra by $\Walg$. According to \cite[Section 1.2]{HC}, to $\J$ one can assign an irreducible representation
of $\Walg$ defined up to twisting with an outer automorphism. More precisely, pick $e\in \Orb$
and include it into an $\mathfrak{sl}_2$-triple $(e,h,f)$. Let $G$ be the simply connected group
with Lie algebra $\g$. Set $Q:=Z_G(e,h,f)$, it is the reductive part of $Z_G(e)$.
The group $Q$ acts on $\Walg$ by algebra automorphisms, moreover, the action is Hamiltonian
meaning that there is a compatible Lie algebra homomorphism (in fact, an inclusion)
$\mathfrak{q}\hookrightarrow \Walg$. So the component group $\Gamma:=Q/Q^\circ$ acts
on the set  $\Irr_{fin}(\Walg)$ of isomorphism classes of irreducible representations.
The main result of \cite{HC} is that the orbit set $\operatorname{Irr}_{fin}(\Walg)/\Gamma$
is naturally identified with the set $\Prim_{\Orb}(\U)$ of primitive ideals in $\U$
corresponding to the orbit $\Orb$. The papers \cite{LO,BL} explain how to compute the
$A$-orbits corresponding to the primitive ideals, see \cite[Theorem 1.1]{LO} for the case of integral central character and
and \cite[Theorem 5.2, Corollary 5.3]{BL} for the full generality.

\subsection{Known results}
In particular, to $\J$ we can assign another numerical invariant, the dimension of the corresponding
finite dimensional irreducible representation of $\Walg$, denote this representation by $V$.
There is a lot of evidence that $\dim V$ and $\Grk(\J)$ are closely related. For example,
it was shown in \cite{W_quant} that $\Grk(\J)\leqslant \dim V$. Premet improved this result
in \cite{Premet_Goldie}, where he proved that $\dim V$ is divisible by $\Grk(\J)$.
On the other hand, in \cite[Corollary 1.2]{nilp_quant} the first named author proved that
$\Grk(\J)=\dim V$ provided $\J$ has integral central character for all orbits but one
in type $E_8$ (and whether the equality holds for the remaining orbit is not known).
One also has $\Grk(\J)=\dim V$ when $\g=\mathfrak{sl}_n$, see
\cite{Premet_Goldie} and \cite{Brundan}.
Another main result of \cite{W_dim} is a Kazhdan-Lusztig type formula
for $\dim V$ (for $V$ with integral central character). So the equality $\Grk(\J)=\dim V$
yields a formula  for $\Grk(\J)$.

On the other hand, Premet in \cite[Remark 4.3]{Premet_Goldie} found a series of examples of primitive ideals $\J$, where the Goldie
rank is always $1$, while the dimension can be arbitrarily large. We will revisit that example in our paper.

\subsection{Main result}
The main result of this paper is the inequality $\Grk(\J)\leqslant \dim V/d_V$, where $d_V$
is a positive integer determined as follows. Let $Q_V$ denote the stabilizer of (the isomorphism class of)
$V$ in $Q$. Note that $Q_V$ is a finite index subgroup in $Q$ because  $Q^\circ$
acts trivially on the set of isomorphism classes of finite dimensional irreducible $\Walg$-modules.  Since $Q$ acts on $\Walg$ by automorphisms and $V$ is an irreducible $\Walg$-module,
we see that $V$ is a projective representation of $Q_V$, let $\psi$ denote the Schur multiplier,
the class in $H^2(Q_V,\C^\times)$ measuring the failure of $V$ to be a genuine representation.

\begin{defi}
Define  $d_V$ as the GCD of the dimensions of the projective representations of
$Q_V$ with Schur multiplier $\psi$.
\end{defi}

\begin{Thm}\label{Thm:main}
We have $\Grk(\J)\leqslant \dim V/d_V$.
\end{Thm}

\begin{Conj}\label{Conj:main}
We have $\Grk(\J)= \dim V/d_V$, at least for classical Lie algebras.
\end{Conj}

Let us make a few remarks about $d_V$. In type A,  we have $d_V=\{1\}$,
this is easily seen once one knows that $\Grk(\J)= \dim V$ but can also
be seen directly. In types $B,C,D$, $Q_V$ is a finite index subgroup
in the product of orthogonal and symplectic groups. One can show that
$d_V$ is a power of $2$, compare with the computations in Sections
\ref{SS_index_comput}.

A current work in progress of the first named author and Bezrukavnikov should produce Kazhdan-Lusztig
type formulas for $\dim V$ (for $V$ with an arbitrary central character).
Together with Conjecture \ref{Conj:main} this should give Kazhdan-Lusztig type formulas
for Goldie ranks.

\subsection{Ideas of proof and structure of the paper}
Let us explain how Theorem \ref{Thm:main} is proved.

The first step is as follows.
Let $H$ denote the preimage of $Q_V$ under the natural epimorphism $Z_G(e)\twoheadrightarrow Q$.
For a projective $H$-module $V'$ with Schur multiplier $\psi$, we can form the equivariant Azumaya algebra
$A=G\times^H \operatorname{End}(V')$ on $G/H$. Our first important
result is to show that the index $\mathsf{ind}(A)$ of $A$ (we recall the definition of the index in
Section \ref{SS_index_main}) coincides with  $d_{V'}$.

The second step is as follows.  We can view $V$ as a projective $H$-representation with
Schur multiplier $\psi$, where the action of $H$ is inflated from $Q_V$.
Our second step is to produce a $G$-equivariant sheaf of $\C[[\hbar]]$-algebras
$\mathcal{A}_\hbar$ on $G/H$ (that is, in a suitable sense, the microlocalization of
the Rees algebra of $\U/\J$) that modulo $\hbar$ reduces to
the equivariant Azumaya algebra $A:=G\times^H \operatorname{End}(V)$.

In the third and final step we use the sheaf $\A_\hbar$ to show that $\Grk(\J)$ cannot exceed the Goldie rank of $A$
that equals $\dim V/d_V$.

The paper is organized as follows. In Section \ref{S_index} we complete step 1 above  computing
the index of an equivariant Azumaya algebra on a homogeneous space. In Section \ref{S_ineq}
we complete the proof of Theorem \ref{Thm:main}. Then in Section \ref{S_example}
we revisit Premet's example mentioned above.

{\bf Acknowledgements}.
Ivan Losev has been funded by the  Russian Academic Excellence Project '5-100'
and also supported by the NSF under grant DMS-1501558. Ivan Panin gratefully acknowledges the support of RFBR grant  19-01-00513-a.

\section{Indexes of equivariant Azumaya algebras}\label{S_index}

\subsection{Main result on indexes}\label{SS_index_main}
Let $G$ be a simply connected algebraic group over $\C$ and $H\subset G$ an algebraic subgroup.

Let $X$ be an algebraic variety. Recall that an Azumaya algebra on $X$ is a coherent sheaf of
algebras on $X$ subject to the following two properties: it is a vector bundle and every fiber
is a matrix algebra. If $G$ acts on $X$ then we can talk about $G$-equivariant Azumaya algebras:
these are $G$-equivariant vector bundles where $G$ acts by isomorphisms of sheaves of algebras.

Let $A$ be a $G$-equivariant   Azumaya algebra on $X:=G/H$. Recall that the specialization
 $A_{\C(X)}$ of $A$ to the generic point of $X$ is a central simple algebra, hence
 a matrix algebra over a skew-field. We are interested in computing
the {\it index} $\ind(A)$. Recall that, for an Azumaya algebra $A$ on an integral scheme
$X$, by the index of $A$ we mean $\sqrt{\dim_{\C(X)} D}$, where $D$ is a skew-field such that
the specialization $A_{\C(X)}$ of $A$ to the generic point of $X$ is a matrix algebra over $D$.

Since $A$ is $G$-equivariant, it is a homogeneous vector bundle on $G/H$. Since $A$
is a sheaf of algebras, the fiber over $1H/H$ is an algebra with an action of $H$
by automorphisms. And since $A$ is  Azumaya, this fiber is of the form
$\End(V)$, where $V$ is a vector space such that $H$ acts on $\operatorname{End}(V)$
by algebra automorphisms. In other words, $V$ is a projective representation
with Schur multiplier, say, $-\psi\in H^2(H,\C^\times)$. Let $\Rep^{\psi}(H)$
denote the category of projective representations of $H$ with Schur multiplier $\psi$
and let $d(\psi)$ denote the GCD of the dimensions of the representations in
$\Rep^\psi(H)$. We also write $d_V$ for $d(\psi)$.

The main result of this section is as follows.

\begin{Thm}\label{Thm:index}
In the notation above, we have $\ind(A)=d(\psi)$.
\end{Thm}

The proof occupies Sections \ref{SS_index_K}-\ref{SS_index_completion}.
In Section \ref{SS_index_comput} we provide some examples of computation that
will become relevant for us in Section \ref{S_example}.

\subsection{K-theoretic definition of the index}\label{SS_index_K}
Let us recall an equivalent definition of the index. Let $A'$ be a central simple algebra
over a field $\K$.  Let $\tilde{\K}$ be an extension of $\K$ such that
the base change $\widetilde{A}':=\tilde{\K}\otimes_{\K}A'$ splits (note that we do not require that
$\tilde{\K}$ is an algebraic extension). Then we have the base change map $K_0(A'\operatorname{-mod})
\rightarrow K_0(\widetilde{A}'\operatorname{-mod})$.  Note that both $K_0$ groups are isomorphic
to $\Z$ (since $A'$ is a central simple algebra,
we have $A'=\operatorname{Mat}_n(D)$, where $D$ is a skew-field; then $A'\operatorname{-mod}$
is equivalent to the category of right vector spaces over $D$, whose $K_0$
is $\mathbb{Z}$).

The following lemma is classical, but we provide a proof for readers convenience.

\begin{Lem}\label{Lem:index_K_theory}
The image of $K_0(A'\operatorname{-mod})$ in $K_0(\widetilde{A}'\operatorname{-mod})\cong \Z$
is $\ind(A')\Z$.
\end{Lem}
\begin{proof}
Let $\widetilde{A}'=\operatorname{Mat}_{m}(\tilde{\K})$ so that $m^2=n^2 \dim_{\K}D$.
The group $K_0(A'\operatorname{-mod})$ is generated by $D^n$, while
$K_0(\widetilde{A}'\operatorname{-mod})$ is generated by $\tilde{\K}^m$.
So the image of interest is the subgroup $k\Z$, where
$$k=\dim_{\tilde{\K}}\tilde{\K}\otimes_{\K}D^n/m=\frac{n}{m}\dim_{\K}D=
\sqrt{\dim_{\K}D}=\operatorname{ind}(A').$$
\end{proof}

Let us return to the situation when we have an equivariant Azumaya algebra $A$ over $G/H$,
$\K=\C(G/H)$ and $A'=A_{\C(G/H)}$. Let us explain our choice of $\tilde{\K}$. Let $\pi$
denote the projection $G\twoheadrightarrow G/H$.

\begin{Lem}\label{Lem:splitting_field}
The pull-back $\pi^*A$ (an Azumaya algebra on $G$) splits. In particular,
for $\tilde{\K}:=\C(G)$, the algebra $\widetilde{A}'$ splits.
\end{Lem}
\begin{proof}
The pullback $\pi^*A$ is the equivariant Azumaya algebra over $G$
with fiber $\operatorname{End}(V)$ over $1$ in $G$. So it is the trivial Azumaya
algebra.
\end{proof}

Consider the categories $\Coh(G/H,A), \Coh(G,\pi^*A)$ of coherent sheaves of modules
over the corresponding Azumaya algebras. We have the pull-back map
$\pi^*: K_0(\Coh(G/H,A))\rightarrow K_0(\Coh(G,\pi^*A))$. The following statement, that should be thought of as a global analog of Lemma \ref{Lem:index_K_theory}, is an important part of our proof of Theorem \ref{Thm:index}.

\begin{Prop}\label{Prop_index_comput}
We have $K_0(\Coh(G,\pi^*A))\cong \Z$ and $\operatorname{im}\pi^*=\ind(A)\Z$.
\end{Prop}
\begin{proof}
We have the following commutative diagram, where the horizontal maps are pull-backs
and the vertical maps are the specializations to the generic points. As before,
we write $A':=A_{\K}$ and $\widetilde{A}':=\tilde{\K}\otimes_{\K}A'$.

\begin{picture}(100,30)
\put(5,2){$K_0(A'\operatorname{-mod})$}
\put(2,22){$K_0(\Coh(G/H,A))$}
\put(62,2){$K_0(\widetilde{A}'\operatorname{-mod})$}
\put(58,22){$K_0(\Coh(G,\pi^*A))$}
\put(16,20){\vector(0,-1){13}}
\put(71,20){\vector(0,-1){13}}
\put(30,3){\vector(1,0){30}}
\put(36,23){\vector(1,0){20}}
\end{picture}

The claim of the proposition will immediately follow from Lemma
\ref{Lem:index_K_theory} once we know that
\begin{enumerate} \item the  map
$K_0(\Coh(G/H,A))\rightarrow K_0(A'\operatorname{-mod})$ corresponding to
the localization to the generic point is surjective,
\item  and the map $K_0(\Coh(G,\pi^*A))\rightarrow
K_0(\tilde{A}'\operatorname{-mod})$ is an isomorphism.
\end{enumerate}
(1) is straightforward. Let us explain why (2) holds.

Since the Azumaya algebra $\pi^*A$ splits, (2) boils
down to showing the map $K_0(\Coh(G))\rightarrow \Z$ (sending the class of a
coherent sheaf on $G$ to the dimension of its fiber at the generic point) is an isomorphism.
This follows because $G$ is simply connected as will be explained
after Proposition \ref{Prop:forget_surjective}.
\end{proof}

\subsection{Equivariant $K_0$-groups}
We can also consider the categories $\Coh^G(G/H,A)$ and $\Coh^G(G,\pi^*A)$
of $G$-equivariant sheaves of $A$- and $\pi^*A$-modules. We still have the pull-back functor
$\pi^*: \Coh^G(G/H,A)\rightarrow \Coh^G(G,\pi^*A)$. This gives rise
to the pull-back map on the level of $K_0$-groups that will be denoted by $\pi^*$ as well.
We will need to describe the image.

\begin{Prop}\label{Prop:equiv_pullback}
We have identifications $$K_0(\Coh^G(G/H,A))\xrightarrow{\sim} K_0(\mathsf{Rep}^\psi(H)),K_0(\Coh^G(G,\pi^*A))\xrightarrow{\sim} \Z$$ so that the map $\pi^*$ sends
the class of $U\in \mathsf{Rep}^\psi(H)$ to $\dim U$.
\end{Prop}
\begin{proof}
For any algebraic subgroup $H^1\subset G$ and any $G$-equivariant Azumaya algebra
$A^1$ on $G/H^1$, the algebra $A^1$ is the homogeneous bundle on $G/H^1$
with fiber $A^1_1$ at the point $1H^1/H^1\in G/H^1$.
We have a category equivalence $\Coh^G(G/H^1,A^1)\rightarrow A^1_{1}\operatorname{-mod}^{H^1}$,
where  the notation $\operatorname{mod}^{H^1}$ means the category of $H^1$-equivariant modules. This equivalence is given by taking the fiber at $1H^1/H^1$, its quasi-inverse sends an $A^1_1$-module $V^1$
to the homogeneous bundle on $G/H^1$ with fiber $V^1$.

Applying this to $H^1:=\{1\}, A^1:=\pi^*A$, we get an equivalence $\Coh^G(G,\pi^*A)\cong (\pi^*A)_1\operatorname{-mod}$.
Recall that $(\pi^*A)_1=\End(V)$, where $V$ is a projective representation of $H$ with Schur multiplier $-\psi$.
Then $\operatorname{Vect}\xrightarrow{\sim} \End(V)\operatorname{-mod}$ via $U\mapsto V\otimes U$.

Similarly,  for $H^1:=H, A^1:=A$, we get $\Coh^G(G/H,A)\cong A_1\operatorname{-mod}^H$.
And we get the equivalence $\Rep^\psi(H)\xrightarrow{\sim} A_1\operatorname{-mod}^H$ by
$U\mapsto V\otimes U$ (since the Schur multipliers of $U$ and $V$ are opposite, $V\otimes U$
is a genuine linear representation of $H$).

Now we claim that, under the resulting identifications, $\Coh^G(G,\pi^*A)\xrightarrow{\sim} \operatorname{Vect}$
and $\Coh^G(G/H,A)\xrightarrow{\sim} \Rep^\psi(H)$ the pull-back functor
$\pi^*: \Coh^G(G/H,A)\rightarrow \Coh^G(G,\pi^*A)$ becomes the forgetful functor
$\Rep^\psi(H)\rightarrow \operatorname{Vect}$. Indeed, the pull-back functor sends
the homogeneous bundle on $G/H$ with fiber $V$ to the homogeneous bundle
on $G$ with fiber $V$.  The claim of the proposition follows.
\end{proof}

\subsection{Forgetful map}
The goal of this section is to prove the following proposition.

\begin{Prop}\label{Prop:forget_surjective}
Let $X$ be a smooth algebraic variety and $G$ be a simply connected algebraic group acting on $X$.
Let $A$ be a $G$-equivariant Azumaya algebra on $X$.
Then the forgetful map $K_0(\Coh^G(X,A))\rightarrow K_0(\Coh(X,A))$ is surjective.
\end{Prop}

We remark that for $A=\mathcal{O}_X$, this is a theorem of Merkurjev, \cite[Thm.40]{Me}.
This, in particular, implies that $K_0(\operatorname{Coh}(G))\xrightarrow{\sim} \Z$ via taking the
generic rank
(this has been already used in the proof of Proposition \ref{Prop_index_comput}):
indeed, the inverse map is given by the forgetful map $\Z\xrightarrow{\sim} K_0(\operatorname{Coh}^G(G))\rightarrow K_0(\operatorname{Coh}(G))$.

We will reduce the proof of Proposition \ref{Prop:forget_surjective} to the case of $A=\mathcal{O}_X$ using the
Severi-Brauer variety of $A$. Recall that this is a variety $\mathsf{SB}_X(A)$ whose $\C$-points are pairs
$(x,J)$, where $x\in X, J\subset A_x$ is a minimal left ideal of $A_x$. It is closed subvariety
in the total space of the bundle over $X$ whose fiber over $x\in X$ is $\operatorname{Gr}(k,A_x)$,
where $k=\sqrt{\dim A_x}$ (note that this bundle is locally trivial in the Zariski topology).
So $\SB_X(A)$
comes with a natural projection $p:\SB_X(A)\twoheadrightarrow X$ and with the tautological
vector bundle $\J$ whose fiber over a point $(x,J)$ is $J$.

Since $J$ is an $A_x$-module, we see that $\mathcal{J}$ is a module over
the Azumaya algebra $p^*A$. Since $\operatorname{rk}J=\sqrt{\operatorname{rk}A}$, we see that
$p^*A$ splits: $p^*A=\mathcal{E}nd_{\mathcal{O}}(\J)$,
here we write $\mathcal{O}$ for the structure sheaf of $\SB_X(A)$. Also,
$p:\SB_X(A)\twoheadrightarrow X$ is a projective bundle (that is locally trivial in
the \'{e}tale topology). In particular, $\SB_X(A)$ is smooth.

Now if $G$ is an algebraic group acting on $X$ and $A$ is $G$-equivariant, then we have a natural
action of $G$ on $\SB_X(A)$, $p$ is $G$-equivariant and $\J$ is a $G$-equivariant vector bundle on
$\SB_X(A)$.

\begin{proof}[Proof of Proposition \ref{Prop:forget_surjective}]
To simplify the notation, let us write $S$ for $\SB_X(A)$. It is well-known that
$K_0(\Coh(X,A))$ splits as a direct summand of $K_0(\Coh(S))$, \cite[Remark 3.3,Examples 3.6(c)]{Pa1},
\cite{Pa2}.

Let us recall how this works. Let us produce maps $\alpha: K_0(\Coh(X,A))\rightleftarrows K_0(\Coh(S)):\beta$
with $\beta\circ \alpha=\operatorname{id}$. Namely, $\alpha$ is induced by the (exact) functor
$\mathcal{F}\mapsto \J^*\otimes_{p^*A}p^*\mathcal{F}:\Coh(X,A)\rightarrow \Coh(S)$,
while $\beta$ is induced by $\mathcal{G}\mapsto Rp_*(\J\otimes_{\mathcal{O}_S}\mathcal{G}):
D^b(\Coh(S))\rightarrow D^b(\Coh(X,A))$. Now observe that the composition
$D^b(\Coh(X,A))\rightarrow D^b(\Coh(S))\rightarrow D^b(\Coh(X,A))$ is isomorphic to
the identity functor. Indeed, since $p^*A=\J\otimes_{\mathcal{O}_S}\J^*$, the composition is
$$
Rp_*(\J\otimes_{\mathcal{O}_S}\J^*\otimes_{p^*A}p^*(\bullet))=Rp_*(p^*(\bullet))=Rp_*(p^*\mathcal{O}_X)\otimes^L_{\mathcal{O}_X}\bullet,
$$
where the last equality is the projection formula. Of course, $p^*\mathcal{O}_X=\mathcal{O}_S$. Since
$p:S\rightarrow X$ is a projective bundle, we see that $Rp_*(\mathcal{O}_S)=\mathcal{O}_X$.
This proves that $\beta\circ\alpha=\operatorname{id}$.

Similarly, we have maps $\alpha_G: K_0(\Coh^G(X,A))\rightleftarrows K_0(\Coh^G(S)):\beta_G$ with
$\beta_G\circ \alpha_G=\operatorname{id}$. Note that the forgetful maps $K_0(\Coh^G(X,A))\rightarrow
K_0(\Coh(X,A))$ and $K_0(\Coh^G(S))\rightarrow K_0(\Coh(S))$ intertwine $\alpha_G$ with $\alpha$ and
$\beta_G$ with $\beta$. The forgetful map $K_0(\Coh^G(S))\rightarrow K_0(\Coh(S))$ is surjective
as was explained after the statement of the proposition. Being a retraction of a surjective map, the forgetful map $K_0(\Coh^G(X,A))\rightarrow
K_0(\Coh(X,A))$ is surjective as well.
\end{proof}

\subsection{Completion of the proof}\label{SS_index_completion}
Let us finish the proof.

\begin{proof}[Proof of Theorem \ref{Thm:index}]
We have the following commutative diagram, where the vertical maps $F_{G/H},F_G$ are the forgetful ones.

\begin{picture}(150,30)
\put(1, 22){$K_0(\Rep^{-\psi}(H))$}
\put(42,22){$K_0(\Coh^G(G/H,A))$}
\put(92,22){$K_0(\Coh^G(G,\pi^*A))$}
\put(42,2){$K_0(\Coh(G/H,A))$}
\put(92,2){$K_0(\Coh(G,\pi^*A))$}
\put(142,2){$\Z$}
\put(28,23){\vector(1,0){13}}
\put(32,24){\tiny $\cong$}
\put(79,23){\vector(1,0){12}}
\put(84,24){\tiny $\pi^*$}
\put(56,21){\vector(0,-1){14}}
\put(57,13){\tiny $F_{G/H}$}
\put(108,21){\vector(0,-1){14}}
\put(109,13){\tiny $F_G$}
\put(77,3){\vector(1,0){14}}
\put(83,4){\tiny $\pi^*$}
\put(125,3){\vector(1,0){15}}
\put(131,4){\tiny $\cong$}
\end{picture}

First, let us show that $F_G: K_0(\Coh^G(G,\pi^*A))\rightarrow K_0(\Coh(G,\pi^*A))$
is an isomorphism. Indeed, since $\pi^*A$ splits in a $G$-equivariant way -- i.e., it is
the endomorphism sheaf of a $G$-equivariant vector bundle, we reduce to
showing that the forgetful map $K_0(\Coh^G(G))\rightarrow K_0(\Coh(G))$ is an isomorphism.
This was established after Proposition \ref{Prop:forget_surjective}.   The isomorphism sends the class
of any vector bundle to its rank. So $F_G$ is an isomorphism as well.

By Proposition \ref{Prop:forget_surjective}, $F_{G/H}$ is surjective. So the image of
$K_0(\Coh(G/H,A))$ in $K_0(\Coh(G,\pi^*A))=\Z$ coincides with the image of
$K_0(\Rep^\psi(H))$ in $K_0(\Coh^G(G,\pi^*A))=\Z$. The latter is $d(\psi)\Z$
by Proposition \ref{Prop:equiv_pullback}. The former is $\ind(A)\Z$ by
Proposition \ref{Prop_index_comput}. The equality $\ind(A)=d(\psi)$ follows.
\end{proof}

\subsection{Examples of computations of $d_V$}\label{SS_index_comput}
Let us provide two examples that will be relevant for what follows.

\begin{Ex}\label{Ex:spinor}
Let $H$ be an algebraic group with reductive part $\operatorname{SO}_{2n+1}$ and $V$ be the spinor representation (of dimension $2^n$) of $\operatorname{SO}_{2n+1}$
(hence of $H$), which is the irreducible
$\mathfrak{h}$-module with highest weight $\omega_n=\frac{1}{2}(\epsilon_1+\ldots+\epsilon_n)$
(in the standard notation). We claim that the dimension of any $H$-module with the same Schur
multiplier is divisible by $2^n$ (and so $d_V=2^n$). This is equivalent to saying that the dimension of
any irreducible $\h$-module $V(\lambda)$ such that $\lambda-\omega_n$ is in the root lattice is divisible
by $2^n$. This will follow if we check that the cardinalities of the Weyl group orbits of
dominant weights $\lambda'$ with $\lambda'-\omega_n$ in the root lattice are divisible by
$2^n$. We have $\lambda'=\sum_{i=1}^n (m_i+\frac{1}{2})\epsilon_i$, where $m_1\geqslant m_2\geqslant\ldots
\ldots m_n\geqslant 0$ are integers. The stabilizer $W_{\lambda'}$ is included into $S_n\subset W$
so  $|W\lambda'|$ is indeed divisible by $|W|/|S_n|=2^n$.
\end{Ex}

\begin{Ex}\label{Ex:half_spinor}
Now let the reductive part be $\operatorname{SO}_{2n}$ and $V$ be one of the half-spinor representations (of dimension $2^{n-1}$).
Similarly to the previous example, we see that $d_V=2^{n-1}$.
\end{Ex}

Note that, in general, it is not true that $d_V=\dim V$ for a minuscule representation of
$\h$. For example, for $\h=\mathfrak{sl}_n$, the dimension of $S^2(\C^n)$ is not divisible by
that of the corresponding minuscule representation $\Lambda^2(\C^n)$.

\section{Inequality on Goldie ranks}\label{S_ineq}
\subsection{W-algebras}
Let us recall some results about W-algebras, see \cite{Premet1,HC}.

Let $G$ be a simply connected semisimple algebraic group, $\g$ its Lie algebra, $\Orb\subset \g$
a nilpotent orbit, $e\in \Orb$. We will write  $Q$ for $Z_G(e,h,f)$.

Pick an $\slf_2$-triple $(e,h,f)$. Set $S=e+\mathfrak{z}_\g(f)$,
this is a transverse slice to $\Orb$. The group $Q$ naturally acts on $S$. Also we have a
$\C^\times$-action on $S$. Namely, we introduce a grading on $\g$, $\g=\bigoplus_i \g(i)$ by eigenvalues
of $[h,\cdot]$. We define a $\C^\times$-action on $\g$ by $t.x:=t^{i-2}x$ for $x\in \g(i)$. Clearly, the
action fixes $S$ giving rise to a grading on $\C[S]$.
The algebra $\C[S]$ admits a Poisson bracket of degree $-2$, see,
e.g., \cite[Section 3]{GG}.

A finite W-algebra $\Walg$ as constructed by Premet in \cite{Premet1} (see also
\cite{W_quant} for an equivalent alternative definition) is a filtered quantization
of the graded Poisson algebra $\C[S]$, i.e., $\Walg$ is a $\Z_{\geqslant 0}$-filtered
associative algebra and we have an isomorphism $\gr\Walg\cong \C[S]$ of graded
Poisson algebras. The group $Q$ acts on $\Walg$ by filtered algebra automorphisms
and this action is Hamiltonian: we have a $Q$-equivariant inclusion $\q\hookrightarrow \Walg$
such that the adjoint action of $\q$ on $\Walg$ coincides with the differential of
the $Q$-action.

In what follows we will need an isomorphism of completions
from \cite{W_quant,HC} that connects $\U$ and $\Walg$. Namely, consider the Rees algebra
$\U_\hbar$ of $\U$ with $\deg \g=2$. Let $\chi\in \g^*$ be the image of $e$ under the isomorphism
$\g\cong \g^*$ coming from the Killing form. The element  $\chi$ gives rise to the composed homomorphism $\U_\hbar
\twoheadrightarrow \C[\g^*]\twoheadrightarrow \C$ that we also denote by
$\chi$. Consider the completion $\U_\hbar^{\wedge_\chi}:=\varprojlim_{n\rightarrow \infty}
\U_\hbar/ (\ker\chi)^n$. This is a complete and separated topological $\C[[\hbar]]$-algebra
that is flat over $\C[[\hbar]]$. Note that it carries an action of $H$ by algebra automorphisms
and this action is Hamiltonian in the sense  that the differential of the $H$-action
coincides with the map $x\mapsto \hbar^{-2}[x,\cdot]: \h\rightarrow \operatorname{Der}(\U_\hbar^{\wedge_\chi})$.
There is also a $\C^\times$-action on $\U_\hbar^{\wedge_\chi}$ coming from the grading on $\g$
with $\deg \g(i)=i+2$.

Then, see, e.g., \cite[Section 2.3]{HC}, there is a $Q\times\C^\times$-equivariant decomposition
\begin{equation}\label{eq:decomp} \U_\hbar^{\wedge_\chi}\cong
\Weyl_\hbar^{\wedge_0}\widehat{\otimes}_{\C[[\hbar]]}\Walg_\hbar^{\wedge_\chi},
\end{equation} where
the notation is as follows. We consider the vector space $[\g,f]$. This space is symplectic
with form $\omega(v_1,v_2):=\langle\chi, [v_1,v_2]\rangle$ so we can form its Weyl algebra
$\Weyl$. Let $\Weyl_\hbar$ denote the Rees algebra of $\Weyl$, and let $\Weyl_\hbar^{\wedge_0}$
be the completion of $\Weyl_\hbar$ at $0$. Similarly, $\Walg_\hbar$ is the Rees algebra of $\Walg$ and
$\Walg_\hbar^{\wedge_\chi}$ is the completion of the former. The action of $Q\times \C^\times$
on the right hand side of (\ref{eq:decomp}) is diagonal, we will not need a precise description
of the action on $\Weyl_\hbar^{\wedge_0}$.

This construction was used in \cite{HC} to establish a bijection $\operatorname{Prim}_{\Orb}(\U)
\cong \operatorname{Irr}_{fin}(\Walg)/\Gamma$. Namely, let $\J$ be a
two-sided ideal in $\U$. We can consider the corresponding Rees ideal $\J_\hbar$ in $\U_\hbar$
and its closure $\J_\hbar^{\wedge_\chi}$ in $\U_\hbar^{\wedge_\chi}$. Then there is a unique
two-sided ideal $\J_{\dagger}\subset \Walg$ such that $$\J_\hbar^{\wedge_\chi}\cong \Weyl_\hbar^{\wedge_0}\widehat{\otimes}_{\C[[\hbar]]}\J_{\dagger,\hbar}^{\wedge_\chi},$$
see, e.g.,  \cite[Section 3.4]{W_quant} or \cite[Proposition 3.3.1]{HC}.
If $\overline{\Orb}$ is an irreducible component of the associated variety of $\J$,
then $\J_{\dagger}$ has finite codimension. If $\J$ is, in addition, primitive, then $\J_{\dagger}$
is a maximal $Q$-stable ideal of finite codimension in $\Walg$. Such ideals are in
a natural one-to-one correspondence with the $\Gamma$-orbits in $\operatorname{Irr}_{fin}(\Walg)$.
This gives rise to a bijection $\operatorname{Prim}_{\Orb}(\U)\xrightarrow{\sim}
\operatorname{Irr}_{fin}(\Walg)/\Gamma$ mentioned in the introduction.

Conversely, to a two-sided ideal $\I\subset \Walg$ we can assign a two-sided ideal
$\I^{\dagger}\subset\U$: the maximal two-sided ideal in $\U$ with the property that
$(\I^\dagger)_{\dagger}\subset \I$, see \cite[Section 3.4]{W_quant}.

\subsection{Jet bundles}
Here we will explain various constructions of jet bundles to be used in the construction of
the sheaf of algebras $\A_\hbar$ on $G/H$ in the next section.  Here the meaning of $H$ is the same as
in the introduction: $H$ is a finite index subgroup of $Z_G(e)$.

Below for an algebraic variety $X$ we write $\Str_X$ for its sheaf of regular functions
and $\C[X]$ for its global sections.

Let us start with the usual jet bundle $\jet \Str_X$ of a smooth algebraic variety $X$.
Let $\mathcal{I}_\Delta\subset \mathcal{O}_{X\times X}$ be the sheaf of ideals of the diagonal in $X\times X$. Then, by definition,
$\jet\Str_X$ is the formal completion of $\Str_{X\times X}$ with respect to $\mathcal{I}_\Delta$.
We view $\jet\Str_X$ as a sheaf  on $X$ via push-forward with respect to the projection $p_1:X\times X
\rightarrow X$ to the first copy. This sheaf is pro-coherent (the inverse limit of coherent sheaves),
namely, $\jet\Str_X=\varprojlim p_{1*}(\Str_{X\times X}/I_\Delta^n)$.
Note that the fiber of
$\jet\Str_X$ over $x\in X$ is the formal completion of the stalk $\Str_{X,x}$.
The tangent sheaf of Lie algebras $T_X$ acts on $\Str_{X\times X}$ via differentiation in
the first copy.
This gives rise to a flat connection on $\jet \Str_X$.
The sheaf of  flat sections $(\jet \Str_X)^\nabla$
is identified with $\Str_X$: it sits inside $\jet\Str_X$ as $p_2^* \Str_X$,
where $p_2:X\times X\rightarrow X$ is the projection to the second
copy. This construction trivially extends to the jet bundle
$\jet A$ of an Azumaya algebra $A$ on $X$: we complete the sheaf $\Str_X\otimes A$
on $X\times X$ and push the completion to the first copy of $X$.
By the construction,  we have a homomorphism of sheaves
of algebras with flat connection $\jet\Str_X\rightarrow \jet A$.

Now suppose that $X$ is a symplectic variety that comes with a Hamiltonian action of $G$.
We are going to define a pro-coherent sheaf $\jet \U_{\hbar,X}$ on $X$ whose fiber at $x\in X$
is the completion of $\U_\hbar$ at $\mu(x)$, where $\mu:X\rightarrow \g^*$ is the moment map.
Namely, note that we have a $G$-equivariant homomorphism of sheaves $\mathcal{O}_X\otimes \g
\rightarrow \mathcal{O}_{X}\otimes \C[X], f\otimes \xi\mapsto f\otimes \mu^*(\xi)$. It extends to a homomorphism
of sheaves of algebras $\mathcal{O}_X\otimes S(\g)\rightarrow \mathcal{O}_{X}\otimes \C[X]$.
The epimorphism $\U_\hbar\twoheadrightarrow S(\g)$ then gives rise to
a homomorphism $\mathcal{O}_X\otimes \U_\hbar\rightarrow \mathcal{O}_{X}\otimes \C[X]$
of sheaves of algebras on $X$. Inside $\mathcal{O}_X\otimes \C[X]$ we have the ideal
of the diagonal in $X\times X$, this ideal is generated by $f\otimes 1-1\otimes f$
for $f\in \C[X]$.
Let $\mathcal{I}_{\mu,\Delta}\subset \mathcal{O}_X\otimes \U_\hbar$
denote the pre-image of the ideal sheaf of the diagonal.
We set $\jet \U_{\hbar,X}:=\varprojlim_{n\rightarrow \infty} \mathcal{O}_X\otimes \U_\hbar/\mathcal{I}_{\mu,\Delta}^n$.
This sheaf again carries a natural flat connection $\nabla$ extended by continuity from the trivial
connection on $\Str_X\otimes \U_\hbar$. Note that $\nabla$ satisfies the following identity,
where $\xi_X$ stands for the velocity vector field induced by $\xi\in \g$:
\begin{equation}\label{eq:connection}
\nabla_{\xi_X}=\xi_{\jet\U_\hbar}-\frac{1}{\hbar^2}[\xi,\bullet].
\end{equation}
On the right hand side, the notation is as follows. We write $\xi_{\jet\U_\hbar}$ for the derivative
of the $G$-equivariant sheaf $\jet\U_\hbar$ induced by $\xi$. In the bracket $\xi$ is viewed as
an element of $\U_\hbar\hookrightarrow \Str_X\otimes \U_\hbar$. The reason why (\ref{eq:connection})
holds is that it holds on $\Str_X\otimes \U_\hbar$ -- there (\ref{eq:connection}) just says
that the action of $\g$ on $\Str_X\otimes \U_\hbar$ is diagonal   -- and then extends to $\jet \U_\hbar$
by continuity.

Note that we have a homomorphism of sheaves with flat connections $\Str_X\otimes \C[X]=p_{1*}\Str_{X\times X} \rightarrow \jet \Str_X$. This gives rise to a homomorphism $\Str_X\otimes \U_\hbar
\rightarrow \jet\Str_X$ (that sends $\hbar$ to $0$). This homomorphism is continuous
and is compatible with flat connections.
So it lifts to  a homomorphism of sheaves of algebras
with flat connections $\jet \U_{\hbar,X}\rightarrow \jet\Str_X$. Note that $\jet \U_{\hbar,X}\twoheadrightarrow \jet\Str_X$
when $X$ is a homogeneous $G$-space. This is because the homomorphism
$\Str_X\otimes \g\rightarrow T_X$  induced by the action of
$G$  is surjective.  We also observe that the fiber of $\jet\U_\hbar$
at $x\in X$ is the completion $\U_\hbar^{\wedge_{\mu(x)}}$.

We will apply these constructions in the case when $X=G/H$, where $H$ is a finite index subgroup of $Z_G(e)$.
Here $\jet\U_\hbar$ is the $G$-homogeneous pro-vector bundle whose fiber at $1H$ is $\U_\hbar^{\wedge_\chi}$.

\subsection{Sheaves $\Afrak_\hbar,\A_\hbar$}
Let us make a remark regarding finite dimensional representations of $\Walg$. Let $V$ be
a finite dimensional $\Walg$-module that also comes with a projective representation
of a finite index subgroup $Q_V\subset Q$ whose differential is the linear representation
of $\q$ that comes from restricting the $\Walg$-action to $\q$. Let $H$ be the preimage of
$Q_V$ in $Z_G(e)$.

Take the trivial filtration
on $V$ and form the Rees $\C[\hbar]$-module $V_\hbar$ and its completion $V_\hbar^{\wedge_0}$.
Then $\underline{\A}_\hbar^{\wedge_0}:=\End_{\C[[\hbar]]}(V_\hbar^{\wedge_0})$ is an algebra
that comes with a $Q_V\times \C^\times$-action by automorphisms and with a $Q_V\times \C^\times$-equivariant
$\C[[\hbar]]$-algebra homomorphism $\Walg_\hbar^{\wedge_\chi}\rightarrow \underline{\A}_\hbar^{\wedge_0}$.
Set $\A_\hbar^{\wedge_\chi}:=\Weyl_\hbar^{\wedge_0}\widehat{\otimes}_{\C[[\hbar]]}\underline{\A}_\hbar^{\wedge_0}$.
This algebra carries a $Q_V\times \C^\times$-action by automorphisms and, thanks to (\ref{eq:decomp}),
comes with a $Q_V\times \C^\times$-equivariant $\C[[\hbar]]$-algebra homomorphism
$\U_\hbar^{\wedge_\chi}\rightarrow \A_\hbar^{\wedge_\chi}$. This homomorphism allows to extend
a $Q_V$-action on $\A_\hbar^{\wedge_\chi}$ to an $H$-action because the action of
$H$ on $\U_{\hbar}^{\wedge_\chi}$ is Hamiltonian.

So we can form a $G$-homogeneous pro-vector bundle $\Afrak_\hbar$ on $G/H$ with fiber
$\A_\hbar^{\wedge_\chi}$. It comes with a flat connection $\nabla$ given by a formula
similar to (\ref{eq:connection}), namely:
\begin{equation}\label{eq:connection1}
\nabla_{\xi_X}=\xi_{\Afrak_\hbar}-\frac{1}{\hbar^2}[\xi,\bullet].
\end{equation}
So we get a $G\times \C^\times$-equivariant homomorphism of $\C[[\hbar]]$-algebras
$\jet\U_\hbar\rightarrow \Afrak_\hbar$ that intertwines the flat connections.

Now let us describe the bundle $\Afrak_\hbar/(\hbar)$ with a flat connection.
Consider the equivariant Azumaya algebra $A:=G\times^H \underline{A}$, where
$\underline{A}:=\underline{\A}_\hbar^{\wedge_0}/(\hbar)(=\operatorname{End}(V))$.

\begin{Lem}
The sheaf of algebras $\Afrak:=\Afrak_\hbar/(\hbar)$ with a flat connection is isomorphic
to $\jet A$.
\end{Lem}
\begin{proof}
We start by obtaining a formula for the connection on $\Afrak$.
Note that we have a  bracket map $Z(\Afrak)
\otimes \Afrak\rightarrow \Afrak$ thanks to the presence of the deformation
$\Afrak_{\hbar}$. Here we write $Z(\Afrak)$ for the center of $\Afrak$, of course,
this is $\jet\Str_X$.

Since we take the trivial filtration on $V$, we see that
the algebra homomorphism $\C[S]^{\wedge_\chi}\rightarrow \underline{A}$ factors through the residue field
of $\C[S]^{\wedge_\chi}$.
Therefore the image of $\jet S(\g)\rightarrow \Afrak$ lies in the
center of $\Afrak$ and, in particular, the subalgebra of flat sections,
$S(\g)\subset \jet S(\g)$ lies in $Z(\Afrak)$. Because of this we get the bracket map
$\g\otimes \mathfrak{A}\rightarrow \mathfrak{A}$. Considering (\ref{eq:connection1})
modulo $\hbar$,  we see that the connection on $\Afrak$ is also given by
\begin{equation}\label{eq:connection3}
\nabla_{\xi_X}=\xi_{\Afrak}-\{\xi,\bullet\}.
\end{equation}

Set $T:=[\g,f]$, recall that this is a symplectic vector space. It is identified with the tangent space
$T_{1H}(G/H)$.
Now we observe that $\Afrak$ is the homogeneous sheaf of algebras $G\times^{H}(\C[[T]]\otimes \underline{A})$ on $G/H$, where the action of
$H$ on $\C[[T]]\otimes \underline{A}$ is diagonal with the natural action on $\C[[T]]$
coming from the identification $T=T_{1H}(G/H)$ and the action on $\underline{A}$ that factors
through the natural action of the quotient $Q_V$ of $H$.

Now we observe that  the sheaf of algebras with a flat connection $\jet A$ has the same description.
Namely, $\jet A$ is the homogeneous vector bundle of algebras with fiber $A^{\wedge_{1H}}$ at $1H\in G/H$.
This fiber $H$-equivariantly identifies with $\C[[T]]\otimes \underline{A}$: we have an embedding
$\C[[T]]\hookrightarrow A^{\wedge_{1H}}$ coming from $\jet\Str_X\hookrightarrow \jet\A$ and also
an embedding $\underline{A}\hookrightarrow A^{\wedge_{1H}}$, together they give rise to the
required isomorphism $\C[[T]]\otimes \underline{A}\xrightarrow{\sim} A^{\wedge_{1H}}$. So
we get a $G$-equivariant sheaf of algebras isomorphism
\begin{equation}\label{eq:jet_iso}\Afrak\xrightarrow{\sim} \jet A.
\end{equation}

Now note that, by the construction, the flat connection on $\jet A$ satisfies
\begin{equation}\label{eq:connection31}
\nabla_{\xi_X}=\xi_{\jet A}-\xi_A,
\end{equation}
where we write $\xi_A$ for the derivation of $\jet A$ coming from the derivation in the second
factor of $\Str_X\boxtimes A$.

We claim that (\ref{eq:jet_iso}) intertwines the flat connections, which will finish the proof.
Since the isomorphism is $G$-equivariant, it is enough to show that it intertwines $\{\xi,\bullet\}$
with $\xi_A$. Note that thanks to (\ref{eq:connection3}) and (\ref{eq:connection31}),
both $\{\xi,\bullet\},\xi_A$ are $\Str_{G/H}$-linear. Thanks to the $G$-equivariance
of both $\{\xi,\bullet\},\xi_A$, it is enough to check that they coincide on the fiber at $1H$,
which is $A^{\wedge_{1H}}$. On that fiber both maps are the derivatives induced by the
$G$-equivariant structure on $A$.
\end{proof}

So we see that $\Afrak_\hbar$ is a formal deformation of $\jet A$.
Set $\A_\hbar=\Afrak_\hbar^{\nabla}$, the sheaf of flat sections. This is a $G$-equivariant sheaf of $\C[[\hbar]]$-algebras
on $G/H$ (but not a sheaf of $\Str_{G/H}$-modules).

\begin{Prop}\label{Prop:flat_deformation}
The sheaf $\A_\hbar$ is a formal deformation on $A$.
\end{Prop}
\begin{proof}
Note that $\A_\hbar$ is a subsheaf in $\Afrak_\hbar$, hence it is flat over $\C[[\hbar]]$.
The subsheaf $\A_\hbar$ is closed in the $\hbar$-adic topology. Hence it is complete and separated in the $\hbar$-adic topology. So we only need  to check
that $\A_\hbar/(\hbar)=A$.  Note that, since taking the flat sections
is a left exact functor, we have $\A_\hbar/(\hbar)\hookrightarrow A$.

Note that $(\A_\hbar/(\hbar))^{\wedge_{1H}}=\A_\hbar^{\wedge_{1H}}/(\hbar)$.
Since $\A_\hbar$ is $G$-equivariant, we have that $\A_\hbar/(\hbar)\hookrightarrow A$
is an isomorphism if and only if
\begin{equation}\label{eq:comp_iso}\A_\hbar^{\wedge_{1H}}/(\hbar)\xrightarrow{\sim} A^{\wedge_{1H}}.
\end{equation}
Note that the right hand side is identified with $\C[[T]]\otimes \underline{A}$.

Now note that $\U_\hbar$ naturally maps to $\Gamma(\A_\hbar)$. This gives rise to a homomorphism
$\U_\hbar^{\wedge_\chi}\rightarrow \A_\hbar^{\wedge_{1H}}$. In particular, we get a homomorphism
$\Weyl_\hbar^{\wedge_0}\rightarrow \A_\hbar^{\wedge_{1H}}$. It is injective and induces a
decomposition $\A_\hbar^{\wedge_{1H}}/(\hbar)=\C[[T]]\otimes \underline{A}'$, where
$\underline{A}'$ is the quotient modulo $\hbar$  of the centralizer
of $\Weyl_\hbar^{\wedge_0}$ in  $\A_\hbar^{\wedge_{1H}}$.

Note that (\ref{eq:comp_iso})
is $\g$-equivariant. The action of $\g$ induces connections on the $\C[[T]]$-modules
$\A_\hbar^{\wedge_{1H}}/(\hbar), A^{\wedge_{1H}}$ and (\ref{eq:comp_iso})
intertwines these connections.
This reduces the
proof of (\ref{eq:comp_iso}) to checking that  $\dim \underline{A}'=\dim \underline{A}$.

Note that the kernel of
$\U_\hbar\rightarrow \Gamma(\A_\hbar)$ is $\J_\hbar$. So the homomorphism  $\Walg_\hbar^{\wedge_\chi}
\rightarrow \A_\hbar^{\wedge_{1H}}$
factors through $\Walg_\hbar^{\wedge_\chi}/\J_{\dagger,\hbar}^{\wedge_\chi}$. The latter is the direct sum
of matrix algebras over $\C[[\hbar]]$, each of dimension $\dim\underline{A}$. This implies
the equality $\dim \underline{A}'=\dim \underline{A}$ and finishes the proof.
\end{proof}

\subsection{Inequality for Goldie ranks}
The following proposition together with Theorem \ref{Thm:index} prove
Theorem \ref{Thm:main}.

\begin{Prop}\label{Prop:Godlie_ineq}
We have $\Grk(\J)\leqslant \dim V/\mathsf{ind}(A)$.
\end{Prop}
\begin{proof}
The algebra $\U_\hbar/\J_\hbar$ is prime and Noetherian, because $\U/\J$ is so.
Note that the Goldie ranks of $\U/\J$ and $\U_\hbar/\J_{\hbar}$ coincide
(because $(\U_\hbar/\J_{\hbar})[\hbar^{-1}]=(\U/\J)[\hbar^{\pm 1}]$).
For any affine open subset $U\subset G/H$, we have an inclusion
$\U_\hbar/\J_{\hbar}\subset \Gamma(U,\A_\hbar)$. Further, $\Gamma(U,\A_{\hbar})/(\hbar)=
\Gamma(U,A)$. The algebra $\Gamma(U, A)$
is an Azumaya $\C[U]$-algebra hence is Noetherian and prime. The algebra
$\Gamma(U,\A_\hbar)$ is a formal deformation of $\Gamma(U,A)$.
Since $\Gamma(U,A)$ is
a prime Noetherian algebra, the algebra $\Gamma(U,\A_\hbar)$ is prime and Noetherian as well.
Indeed, the usual argument that the algebra of formal power series over a Noetherian ring
is Noetherian generalizes to show that $\Gamma(U,\A_\hbar)$ is Noetherian.
Further, let $I_\hbar,J_\hbar$ be two-sided ideals in $\Gamma(U,\A_\hbar)$. Let $I,J$
be their images in $\Gamma(U,A)$. If $I_\hbar J_\hbar=\{0\}$, then $IJ=\{0\}$. So
one of $I,J$, say $I$ to be definite, is zero. It follows that $I_\hbar\subset \hbar \Gamma(U,\A_\hbar)$.
Now we can replace $I_\hbar$ with $\hbar^{-1}I_\hbar$ and proceed in the same way.
This shows $\Gamma(U,\A_\hbar)$ is prime.

By a result of Warfield, \cite[Theorem 1]{Warfield}, the inclusion  $\U_\hbar/\J_{\hbar}\subset \Gamma(U,\A_\hbar)$
implies the inequality of Goldie ranks $\Grk(\U_\hbar/\J_{\hbar})\leqslant
\Grk(\Gamma(U,\A_\hbar))$. So it remains to show that we can choose $U$ so that $\Grk(\Gamma(U,\A_\hbar))=\dim V/\mathsf{ind}(A)$,
which is the Goldie rank of $\Gamma(U,A)$.

We claim that we can pick $U$ in such a way that $\Gamma(U,A)=\operatorname{Mat}_k(D')$,
where $D'$ is a domain, and $k=\dim V/\mathsf{ind}(A)$. Indeed, for an arbitrary
affine open subset $U$ we have $\C(U)\otimes_{\C[U]}\Gamma(U,A)=\operatorname{Mat}_k(D)$
for a skew-field $D$ over $\C(U)$. The embedding $\operatorname{Mat}_k(\C)\hookrightarrow
\C(U)\otimes_{\C[U]}\Gamma(U,A)$ is defined over $\C[U_1]$ for some principal open affine
subset $U_1\subset U$. Let $\epsilon$ denote the primitive idempotent in $\operatorname{Mat}_k(\C)$.
Then $D':=\epsilon\Gamma(U_1,A)\epsilon$ satisfies $\Gamma(U_1,A)=\operatorname{Mat}_k(D')$.
Recall that $A$ is a vector bundle. So the natural map $\Gamma(U_1,A)\rightarrow
\C(U_1)\otimes_{\C[U_1]}\Gamma(U_1,A)=\operatorname{Mat}_k(D)$ is an inclusion.
Hence $D'\hookrightarrow D$. The latter is an algebra homomorphism, hence $D'$
is a domain. It remains to replace $U$ with $U_1$.

Since $\Gamma(U,\A_\hbar)$
is a formal deformation of $\operatorname{Mat}_k(D')$, the inclusion $\operatorname{Mat}_k(\C)
\hookrightarrow \operatorname{Mat}_k(D')$ lifts to $\operatorname{Mat}_k(\C)\hookrightarrow
\Gamma(U,\A_\hbar)$. Indeed, we can lift the idempotent $\epsilon\in \Gamma(U_1,A)$
to an idempotent $\epsilon_\hbar \in \Gamma(U_1,\A_\hbar)$. The right $\epsilon_\hbar \Gamma(U_1,\A_\hbar)\epsilon_\hbar$-module $\Gamma(U_1,\A_\hbar)\epsilon_\hbar$ is free
of rank $k$ because the $\epsilon\Gamma(U_1,A)\epsilon$-module $\Gamma(U_1,A)\epsilon$
is so. Also the
algebra homomorphism
$$\Gamma(U_1,\A_\hbar)\rightarrow
\operatorname{End}_{\epsilon_\hbar \Gamma(U_1,\A_\hbar)\epsilon_\hbar}(\Gamma(U_1,\A_\hbar)\epsilon_\hbar)$$
is an isomorphism because
$$\Gamma(U_1,A)\rightarrow
\operatorname{End}_{\epsilon \Gamma(U_1,A)\epsilon}(\Gamma(U_1,A)\epsilon)$$
is an isomorphism.

The algebra $D'_\hbar:=\epsilon_\hbar \Gamma(U,\A_\hbar)\epsilon_\hbar$  is the formal deformation of $D$. So $D'_\hbar$ is a domain,  $\Gamma(U,\A_\hbar)=
\operatorname{Mat}_k(D'_\hbar)$, and the Goldie rank of $\Gamma(U,\A_\hbar)$ equals $k$. This finishes the proof
of $\Grk(\J)\leqslant \dim V/\mathsf{ind}(A)$.
\end{proof}


\section{Example of computation}\label{S_example}
Here we will use Theorem \ref{Thm:main} to
revisit  an example of a completely prime primitive ideal in \cite{Premet_Goldie} that corresponds
to an irreducible representation of $\Walg$ with large dimension.
%


\subsection{Main result}
Let us start by recalling a classical theorem of Duflo.

Let us write $\mathfrak{t}$ for a Cartan subalgebra of $\g$. For $\lambda\in \mathfrak{t}^*$
let $L(\lambda)$ denote the irreducible module in the usual BGG category
$\mathcal{O}$ with highest weight $\lambda-\rho$, where $\rho$ is half the sum
of the positive roots. Let $\J(\lambda)$ denote the annihilator of $L(\lambda)$
in $\U$. As Duflo proved, the ideals $\J(\lambda)$ exhaust the primitive ideals
of $\U$.

Here is the main result of this section.

Let $\g=\mathfrak{sp}_{2n}$. We consider the primitive ideal $\J:=\J(\rho/2)$ in $\U$.
It follows from results of McGovern that $\Grk(\J)=1$, \cite{McGovern}. By \cite[Remark 4.3]{Premet_Goldie},
the dimension of the corresponding representation of $\Walg$ is bigger than $1$.
We will independently show that
the corresponding representation of $\Walg$ restricts to a spinor representation
of $\mathfrak{q}$, hence the inequality in Theorem \ref{Thm:main} becomes an
equality. From here we will deduce the equality $\Grk(\J)=1$.

Here is a complete statement.

\begin{Prop}\label{Prop:example}
Let $\g=\mathfrak{sp}_{2n}$ with $n>2$. Set $\J=\J(\rho/2)$. Then the following are true:
\begin{enumerate}
\item For $n=2m$, the $A$-orbit in $\Irr_{fin}(\Walg)$ corresponding to $\J$ consists of 2 irreducible representations.
Their restrictions to $\q=\mathfrak{so}_n$ are the (non-isomorphic) half-spinor representations.
\item For $n=2m+1$, the $A$-orbit in $\Irr_{fin}(\Walg)$ corresponding to $\J$ consists a single irreducible
representation. Its restriction to $\q=\mathfrak{so}_n$ is the spinor representation.
\item $\Grk(\J)=1$.
\end{enumerate}
\end{Prop}

\subsection{Category $\mathcal{O}$ for $\Walg$}\label{SS_W_OCat}
Let $T_Q\subset Q$ denote the maximal torus and let $\nu:\C^\times\rightarrow Q$ be a one-parameter
subgroup that is generic in the sense that its centralizer in $\g$ coincides with the centralizer
of $T_Q$. The one-parameter subgroup $\nu$ gives rise to the weight decomposition
$\Walg=\bigoplus_{i\in \Z}\Walg_i$. Set $\Walg_{\geqslant 0}:=\bigoplus_{i\geqslant 0} \Walg_i,
\Walg_{>0}:=\bigoplus_{i>0}\Walg_i, \Ca_\nu(\Walg):=\Walg_{\geqslant 0}/
(\Walg_{\geqslant 0}\cap \Walg\Walg_{>0})$.

Following \cite{BGK}, define the category $\mathcal{O}_\nu(\Walg)$ as the full subcategory in the category of finitely
generated $\Walg$-modules consisting of all modules $M$ such that
\begin{enumerate}
\item $M$ admits a weight decomposition with respect to $\mathfrak{t}_Q$:
$M=\bigoplus_{\alpha\in \mathfrak{t}^*} M_\alpha$, where $M_\alpha:=\{m\in M|
xm=\langle\alpha,x\rangle m\}$.
\item $\Walg_{>0}$ acts on $M$ locally nilpotently.
\end{enumerate}
In particular, any finite dimensional irreducible $\Walg$-module lies in $\mathcal{O}$.
We remark that, modulo (1), (2) can be shown to be equivalent to the condition
that the real numbers $\operatorname{Re}\langle \alpha,\nu\rangle$ are bounded
from above.

Note that for $M\in \Walg\operatorname{-mod}$, the annihilator $M^{\Walg_{>0}}$ is
$\Walg_{\geqslant 0}$-stable and the action of $\Walg_{\geqslant 0}$ on
$M^{\Walg_{>0}}$ factors through $\Ca_\nu(\Walg)$. It turns out that if $M$ is simple
in $\mathcal{O}_\nu(\Walg)$, then $M^{\Ca_\nu(\Walg)}$ is a simple $\Ca_\nu(\Walg)$-module.
The assignment $M\mapsto M^{\Walg_{>0}}$ is a bijection between the simples
in $\mathcal{O}_\nu(\Walg)$ and the simple $\Ca_\nu(\Walg)$-modules. For an irreducible
$\Ca_\nu(\Walg)$-module $N$, we write $L_{\nu}(N)$ for the corresponding simple object
in $\mathcal{O}_\nu(\Walg)$.

The structure of the algebra $\Ca_\nu(\Walg)$ was determined in \cite{BGK}, see also
\cite{W_OCat}. Let $\underline{\g}$ denote the centralizer of $\nu$ in $\g$, note that
$e\in \underline{\g}$. Let $\underline{\U}$ denote $U(\underline{\g})$ and
$\underline{\Walg}$ denote the W-algebra for $(\underline{\g},e)$.
Then there is an isomorphism $\Ca_\nu(\Walg)\xrightarrow{\sim}\underline{\Walg}$,
see \cite[Section 4]{BGK}. This isomorphism is $T_Q$-equivariant but it
does not intertwine the quantum comoment maps $\mathfrak{t}_Q\rightarrow
\Ca_\nu(\Walg), \underline{\Walg}$. Rather it induces a shift by a character $\delta$
that we define now. Pick a Cartan subalgebra $\bar{\mathfrak{t}}\subset \g$
that contains $\mathfrak{t}_Q$ and $h$. Let $\Delta\subset \bar{\mathfrak{t}}^*$ be the root system.
Set $\Delta^-:=\{\beta\in \Delta| \langle\beta,\nu\rangle<0\}$.
We set
\begin{equation}\label{eq:delta}
\delta=\frac{1}{2}\sum_{\alpha\in \Delta^-, \langle\alpha,h\rangle=-1}\alpha+\sum_{\alpha\in \Delta^-,
\langle \alpha,h\rangle\leqslant -2}\alpha.
\end{equation}
Then the isomorphism $\Ca_\nu(\Walg)\xrightarrow{\sim}\underline{\Walg}$ restricts to $x\mapsto x-\langle\delta,x\rangle$
on $\mathfrak{t}_Q$.

Now pick a finite dimensional irreducible $\Walg$-module $V$. It lies in
$\mathcal{O}_\nu(\Walg)$ and so there is a unique finite dimensional irreducible
module $\underline{\Walg}$-module $\underline{V}$ with $V=L_\nu(\underline{V})$.
Let $\underline{V}$ correspond to the primitive ideal $\underline{\J}(\lambda)\subset \underline{\U}$
(annihilating the irreducible $\underline{\g}$-module with highest weight
$\lambda-\rho$). It was shown in \cite[Theorem 5.1.1]{W_1dim} that the primitive ideal
in $\U$ corresponding to $V$ is $\J(\lambda)$.

While we cannot read the Schur multiplier for the $Q_V$-action on $V$ from
$\underline{V}$, we can recover that for the $Q^\circ$-action. Namely,
$\tf_Q$ embeds into the center of $\underline{\Walg}$. The action of $\tf_Q$ on $\underline{V}$
under this embedding is  by $(\lambda-\rho)|_{\tf_Q}$. Therefore the action of $\tf_Q$
on $\underline{V}=V^{\Walg_{>0}}$ via the embedding $\tf_Q\hookrightarrow \Walg$ is via
\begin{equation}\label{eq:character}
(\lambda-\rho-\delta)|_{\mathfrak{t}_Q}.
\end{equation}
It follows that the weights of $\mathfrak{t}_Q$ in $V$ are congruent to $(\lambda-\delta)|_{\tf_Q}$.
This allows to recover the Schur multiplier for the projective action of $Q^\circ$ on $V$.

\subsection{Proof of Proposition \ref{Prop:example}}
Let $\g=\mathfrak{sp}_{2n}$. We consider the orbit $\Orb$ corresponding to
the partition $(2^n)$ (where the superscript indicates the multiplicity).
The group $Q$ is $\operatorname{O}_n$. It is easy to see that the codimension
of $\Orb$ in $\g$ is $n^2$.

\begin{proof}
The proof is in several steps.

{\it Step 1}. It is easy to see that $\dim  \VA(\J)=\dim \Orb$. Indeed,
the integral Weyl group for $\rho/2$ has type $B_{\lfloor n/2\rfloor}\times
D_{\lceil n/2\rceil}$. For the integral root system, $\rho/2$ is  dominant (and, of course, integral).
By \cite[Corollary 3.5]{Joseph}, this implies that $\dim \VA(\J)=\dim \g-\dim \g^\vee_{int}$,
where $\g^\vee_{int}$ stands for the integral subalgebra for $\rho/2$ in $\g^\vee=\mathfrak{so}_{2n+1}$.
Since $\dim \g^\vee_{int}=n^2$, the claim in the beginning of the paragraph follows.

{\it Step 2}.
The symplectic form we use to define $\mathfrak{sp}_{2n}$ is $\omega(u,v)=\sum_{i=1}^{n}(u_i v_{2n+1-i}-v_i u_{2n+1-i})$.
Identify $\bar{\tf}$ with $\C^n$ in a standard way, $\bar{\tf}=\operatorname{diag}(x_1,\ldots,x_n,-x_n,\ldots,-x_1)$.
When $n=2m$, $\tf_Q$ is
embedded into $\tf$ as $$\{(x_1,x_1,x_2,x_2,\ldots, x_m,x_m)\},$$ while for
$n=2m+1$, we get $$\tf_Q=\{(x_1,x_1,\ldots, x_m,x_m,0)\}.$$ We choose $\nu$ corresponding
to the vector $(x_1,\ldots,x_m)\in \Z^m$ with $x_1>x_2>\ldots>x_m>0$, it is dominant
for $\g$.

{\it Step 3}.
The subalgebra $\underline{\g}=\mathfrak{z}_\g(\nu)$ is $\mathfrak{sl}_2^n$ and $e$
is a principal nilpotent element in $\underline{\g}$. The W-algebra $\underline{\Walg}$
is, therefore, the center of $\underline{\U}$. Consider the primitive ideal $\underline{\J}(\rho/2)$
for $\underline{\g}$. It is minimal, so it gives rise to an irreducible representation of $\underline{\Walg}$,
to be denoted by $\underline{V}$. Set $V=L_\nu(\underline{V})$.  By \cite[Theorem 5.1.1]{W_1dim},
$\J(\rho/2)=\operatorname{Ann}_{\Walg}(V)^\dagger$. It follows that $\Orb\subset \VA(\J)$
and, since  the codimensions of $\Orb$ and $\VA(\J)$ in $\g$ coincide (both are equal to $n^2$),
we recover the equality $\VA(\J)=\Orb$. Equivalently,
$V$ is finite dimensional.

{\it Step 4}. Let us determine how $\tf_Q$ acts on $\underline{V}$. According to (\ref{eq:character}),
the action is by the character $(-\rho/2-\delta)|_{\tf_Q}$. Below in this step, we will compute this
character explicitly.

Note that $e$ is an even nilpotent element
so $\delta=\sum_{\alpha\in \Delta^-, \alpha(h)\leqslant -2}\alpha$. Note that $\delta|_{\tf_Q}=
\frac{1}{2}(\sum_{\alpha\in \Delta^-, \alpha(h)\neq 0}\alpha)|_{\tf_Q}$ because $\tf_Q$ centralizes
the $\slf_2$-triple $(e,h,f)$. So $(-\rho/2-\delta)|_{\tf_Q}=(\rho/2-\rho_0)|_{\tf_Q}$, where $\rho_0$
is $\frac{1}{2}\sum_{\alpha>0| \alpha(h)=0}\alpha$. The element $h\in \tf_Q$ equals $(1,-1,1,-1,\ldots)$.
The positive roots that vanish on $h$ are of the form $\epsilon_i-\epsilon_j$, where $i-j$  is even,
and $\epsilon_i+\epsilon_j$, where $i-j$ is odd.

It is easy to see that $\rho=n \epsilon_1+(n-1)\epsilon_2+\ldots+\epsilon_n$.
Let $\eta_i$ denote the $i$th coordinate function on $\tf_Q\cong \C^m$.

Consider the case of $n=2m$ first. Here
$$\rho|_{\tf_Q}=(4m-1)\eta_1+(4m-5)\eta_2+\ldots+3\eta_m.$$
We get
\begin{align*} &2\rho_0=(m-1)(\epsilon_1+\epsilon_2)+(m-3)(\epsilon_3+\epsilon_4)+\ldots +(1-m)(\epsilon_{2m-1}+\epsilon_{2m})+
m(\epsilon_1+\ldots+\epsilon_{2m})\\
&=(2m-1)(\epsilon_1+\epsilon_2)+(2m-3)(\epsilon_3+\epsilon_4)+\ldots+
(\epsilon_{2m-1}+\epsilon_{2m}).
\end{align*}
Hence $2\rho_0|_{\tf_Q}=(4m-2)\eta_1+\ldots+ 2\eta_m$.
We conclude that $(\rho/2-\rho_0)|_{\tf_Q}=(\eta_1+\ldots+\eta_m)/2$. This is the highest weight
of a half-spinor representation.

Now consider the case of $n=2m+1$. Here
$$\rho|_{\tf_Q}=(4m+1)\eta_1+(4m-3)\eta_2+\ldots+ 5\eta_m.$$
We get
\begin{align*} &2\rho_0=m\epsilon_1+(m-2)\epsilon_3+\ldots -m\epsilon_{2m+1}+(m-1)\epsilon_2+(m-3)\epsilon_4+\ldots
+(1-m)\epsilon_{m-1}+\\&+m(\epsilon_1+\epsilon_3+\ldots \epsilon_{2m+1})+(m+1)(\epsilon_2+\ldots+\epsilon_{2m})=\\
&=2m(\epsilon_1+\epsilon_2)+(2m-2)(\epsilon_3+\epsilon_4)+\ldots+ 2(\epsilon_{2m-1}+\epsilon_{2m}).
\end{align*}
So $2\rho_0|_{\tf_Q}=4m\eta_1+4(m-1)\eta_2+\ldots+ 4\eta_m$. It follows that $(\rho/2-\rho_0)|_{\tf_Q}=
(\eta_1+\ldots+\eta_m)/2$, the highest weight of the spinor representation.

{\it Step 5}. It remains to prove that $V$ is irreducible over  $\mathfrak{q}$.
Consider a more general situation. Let $\nu$ be as in Section  \ref{SS_W_OCat},
let $\underline{V}$ be an irreducible $\underline{\Walg}$-module such that $V=L_\nu(\underline{V})$
is finite dimensional. Assume that
\begin{itemize}
\item[(i)] $\dim \underline{V}=1$ (this always holds when $e$ is principal in $\underline{\g}$,
which is true in the case of interest for us),
\item[(ii)] The algebra $\mathfrak{q}$ is semisimple.
\item[(iii)] The action of $\tf_Q$ on $\underline{V}$ via $\tf_Q\hookrightarrow \Ca_\nu(\Walg)$ is by
a minuscule weight, say $\omega$.
\end{itemize}
We claim that in this case $V$ is irreducible over $\mathfrak{q}$ (with highest weight $\omega$).
This partially generalizes
\cite[Theorem 5.2.1]{W_1dim}. The proof is  similar but we provide it for readers
convenience.

Note that all $\tf_Q$-weights $\beta$ in $V$ are congruent to $\omega$ modulo the root lattice of
$\mathfrak{q}$ and also $\langle \beta,\nu\rangle\leqslant \langle\omega,\nu\rangle$ with
equality if and only if $\beta$ is a weight of $\underline{V}$. If $V$ is not irreducible
over $\mathfrak{q}$, then there is another highest weight, say $\omega'$. Since $\omega$
is minuscule, we have $\omega'\geqslant \omega$ (meaning that $\omega'-\omega$ is the sum of
positive roots). This contradicts  the inequality $\langle\omega',\nu\rangle\leqslant
\langle\omega,\nu\rangle$ and proves the irreducibility of $V$ over $\mathfrak{q}$.

{\it Step 6}. Let us now complete the proof.
By Step 4, $V$ satisfies condition (iii) of Step 5, while conditions (i),(ii) were established
above in the proof. So $V$ is the spinor representation of $\mathfrak{so}_{2m+1}$ when $n=2m+1$
and one of the half-spinor representations of $\mathfrak{so}_{2m}$ in the case when $n=2m$.
In the former case, $V$ is $Q$-stable because $\underline{V}$ is stable under $N_Q(\tf_Q)$.
In the latter case, $V$ is not stable because of an outer automorphism in $Q$
that  permutes the half-spinor representations. This proves (1),(2) of the proposition.
(3) now follows from Theorem \ref{Thm:main} and Examples \ref{Ex:spinor} (the case
of $n=2m+1$) and \ref{Ex:half_spinor} (the case of $n=2m$).
\end{proof}

\end{document}